\documentclass[oneside,a4paper,10pt]{amsart}

\usepackage{amsmath,amssymb, amsthm}
\usepackage{t1enc,geometry,enumitem}
\usepackage[utf8]{inputenc}
\usepackage{url}
\usepackage[english]{babel}

\theoremstyle{plain} 
\newtheorem{theorem}{Theorem}
\newtheorem{proposition}[theorem]{Proposition}
\newtheorem{lemma}{Lemma}

\newtheorem{corollary}[theorem]{Corollary}

\theoremstyle{definition}

\newtheorem{remark}[theorem]{Remark}

\newcommand*{\R}{\mathbb{R}}
\newcommand*{\N}{\mathbb{N}}

\newcommand{\norm}[1]{\ensuremath{\left\lVert #1 \right\rVert}}
\newcommand{\halmaz}[1]{\left\{\,#1\,\right\}}
\newcommand{\halmazvonal}[2]{\left\{\,#1\mid #2\,\right\}}

\newcommand*{\D}{\mathbb{D}}
\def\difS{\mathsf{Dif\!f}_+^{\infty }(\mathbb{S}^1)}
\def\F{\mathcal F}
\def\Hol{\mathsf{Hol}}
\def\ih#1{\mathfrak{hol}^*_{\, #1}(\mathbb D^2)}
\def\ihm#1{\mathfrak{hol}^*_{\, #1}(M)}
\def\I{\mathcal I}

\def\P{\mathcal P}
\def\X#1{\mathfrak X({#1})}

\newenvironment{packed_enumerate}{
  \begin{enumerate}[topsep=3pt, partopsep=0pt,leftmargin=25pt]
    \setlength{\itemsep}{0pt}
    \setlength{\parskip}{3pt}
    \setlength{\parsep}{3pt}
  }{\end{enumerate}}

\title[The holonomy of projective Randers two-manifolds of constant
curvature]{The holonomy group of projectively flat Randers
  \\
  two-manifolds of constant curvature}

\author[B.~Hubicska]{Bal\'azs Hubicska}
\email{\url{hubicska.balazs@science.unideb.hu}}

\author[Z.~Muzsnay]{Zolt\'an Muzsnay}
\email{\url{muzsnay@science.unideb.hu}}

\address{University of Debrecen, Institute of Mathematics, Pf.~400, Debrecen,
  4002, Hungary}

\thanks{The research was partially supported by the EFOP 3.6.2-16-2017-00015
  project and the EFOP-3.6.1-16-2016-00022 project. The project is co-financed
  by the European Union and the European Social Fund.}



\begin{document}

\begin{abstract}
  In this paper, we investigate the holonomy structure of the most accessible
  and demonstrative 2-dimensional Finsler surfaces, the Randers surfaces.
  Randers metrics can be considered as the solutions of the Zermelo navigation
  problem.  We give the classification of the holonomy groups of locally
  projectively flat Randers two-manifolds of constant curvature. In particular
  we prove that the holonomy group of a simply connected non-Riemannian
  projectively flat Finsler two-manifold of constant non-zero flag curvature is
  maximal and isomorphic to the orientation preserving diffeomorphism group of
  the circle.
\end{abstract}

\maketitle

\section{Introduction}

The holonomy group of a Riemann or Finsler manifold is a natural geometric
object: it is the transformation group generated by parallel translations
along loops with respect to the associated canonical (linear,
resp.~homogeneous) connection.  The Riemannian holonomy groups have been
extensively studied during the second half of the last century but, maybe
because of the computational difficulties, little attention was paid to the
Finslerian case. Although there are similarities, the holonomy properties of
Finsler manifolds can be very different from the Riemannian ones, as recent
results show.  Indeed, the fundamental result of Borel and Lichnerowicz
\cite{borel_lichn_1952} from 1952 claims that the holonomy group of a simply
connected Riemannian manifold is a (finite dimensional) closed Lie subgroup of
the special orthogonal group $SO(n)$.  In contrast with this, in
\cite{muzs_nagy_2011} it has been proven that the holonomy group of an at
least three-dimensional non-Riemannian Finsler manifold of nonzero constant
curvature is not a compact Lie group.  In \cite{MuzsNagy_2015} it has been
shown that there exist large families of projectively flat Finsler manifolds
of constant curvature such that their holonomy groups are not finite
dimensional Lie groups. In \cite{muzsnay_nagy_2014} projective Finsler
manifolds of constant curvature having infinite dimensional holonomy group
have been characterized.  The proofs in the above-mentioned papers
\cite{muzs_nagy_2011, muzsnay_nagy_2014, MuzsNagy_2015} give estimates for the
dimension of tangent Lie algebras of the holonomy group but unfortunately,
they do not give direct information on the structure of the holonomy groups.

In contrast to the Riemannian case where the complete classification is
already known, Finsler holonomy groups are described only for few, very
special classes of metrics: the holonomy of Landsberg manifolds have been
investigated in \cite{kozma_2000} and the holonomy of Berwald manifolds have
been characterized in \cite{szabo_1981}. First examples describing infinite
dimensional Finslerian holonomy groups can be found in
\cite{MuzsNagy_max_2015}. To achieve further significant progress, it would be
very important to investigate systematically the holonomy structure of
different classes of Finsler metrics.

In this article, we consider one of the most accessible and demonstrative
examples for non-Riemannian Finsler manifolds, the Randers manifolds
\cite{bao_chern_shen_2000, randers_1941}. In the Randers case, the Finsler
function is a Riemann norm deformed by a 1-form.  As it has been proven in
\cite{BaoRoblesShen_2004}, Randers metrics describe the Zermelo navigation
problem on Riemannian manifolds. This fact may suggest that the holonomy
structure of Randers manifolds are similar to that of Riemann manifolds but,
on the contrary: they can be very different, as the results of this paper
show.  We focus our attention on the holonomy properties of simply connected,
locally projectively flat Randers two-manifolds of constant flag curvature
$\lambda$.  This class of manifolds was already considered in
\cite{MuzsNagy_2015} where it has been proved that the holonomy group of such
a manifold is finite dimensional if $\lambda=0$ or the metric is Riemannian,
and \emph{infinite dimensional} if $\lambda\neq 0$ and the metric is
non-Riemannian. The finite dimensional holonomy structures are already well
known, but nothing was known about the infinite dimensional case: the results
reveal no information about the structure of the holonomy group.

The goal of this paper is to complete the results of \cite{MuzsNagy_2015} by
describing the infinite dimensional holonomy structure of projectively flat
Randers surfaces of non-zero constant curvature.  The main result (Theorem
\ref{thm:holon-proj-flat}.) shows that the holonomy group of a such manifold
is maximal and its closure is isomorphic to $Dif\!f_+(\mathbb S^1)$, the
orientation preserving diffeomorphism group of the circle $\mathbb{S}^1$.
This result is surprising because it shows that even in the case when the
geodesic structure is simple (the geodesics are straight lines), the holonomy
group can still be a very large group.  Finally, we obtain the classification
of the holonomy groups of projectively flat Randers surfaces (Corollary
\ref{cor:holonomy}).

\bigskip

\section{Preliminaries}

Throughout this article, $M$ is a $C^\infty$ smooth manifold,
${\mathfrak X}^{\infty}(M)$ denotes the Lie algebra of smooth vector fields on
$M$ and ${\mathsf{Diff}}^\infty(M)$ is the group of $C^\infty$-diffeomorphisms
of $M$.  The first and the second tangent bundles of $M$ are denoted by
$(TM,\pi ,M)$ and $(TTM,\tau ,TM)$, respectively.

\subsection{Finsler manifolds} \

\smallskip

\noindent
A \emph{Finsler manifold} is a pair $(M,\mathcal F)$, where the norm
function $\F\colon TM \to \mathbb{R}_+$ is continuous, smooth on $\hat T
M \!:= \!TM\!  \setminus\! \{0\}$, its restriction ${\mathcal
  F}_x={\mathcal F}|_{_{T_xM}}$ is a positively homogeneous function of
degree one and the symmetric bilinear form
\begin{displaymath}
  g_{x,y} \colon (u,v)\ \mapsto \ g_{ij}(x, y)u^iv^j=\frac{1}{2}
  \frac{\partial^2 \mathcal F^2_x(y+su+tv)}{\partial s\,\partial t}\Big|_{t=s=0}
\end{displaymath}
is positive definite at every $y\in \hat T_xM$.
\\[1ex]
\emph{Geodesics} of $(M, \mathcal F)$ are determined by a system of $2$nd
order ordinary differential equation $\ddot{x}^i + 2 G^i(x,\dot x)=0$,
$i = 1,\dots,n$ in a usual local coordinate system $(x^i,y^i)$ of $TM$, where
$G^i(x,y)$ are given by
\begin{equation}
  \label{eq:G_i}  G^i(x,y):= \frac{1}{4}g^{il}(x,y)\Big(2\frac{\partial
    g_{jl}}{\partial x^k}(x,y) -\frac{\partial g_{jk}}{\partial
    x^l}(x,y) \Big) y^jy^k.
\end{equation}
A vector field $X(t)=X^i(t)\frac{\partial}{\partial x^i}$ along a curve
$c(t)$ is said to be parallel with respect to the associated
\emph{homogeneous (nonlinear) connection} if it satisfies
\begin{equation}
  \label{eq:D}
  D_{\dot c} X (t):=\Big(\frac{d X^i(t)}{d t}+  G^i_j(c(t),X(t))\dot c^j(t)
  \Big)\frac{\partial}{\partial x^i}
  =0, 
\end{equation}
where $ G^i_j=\frac{\partial G^i}{\partial y^j}$.
\\[1ex]
The \emph{horizontal Berwald covariant derivative} $\nabla_X\xi$ of
$\xi(x,y) = \xi^i(x,y)\frac {\partial}{\partial y^i}$ by the vector field
$X(x) = X^i(x)\frac {\partial}{\partial x^i}$ is expressed locally by
\begin{equation}
  \label{covder}
  \nabla_X\xi = \left(\frac {\partial\xi^i(x,y)}{\partial x^j} 
    - G_j^k(x,y)\frac{\partial \xi^i(x,y)}{\partial y^k} + 
    G^i_{j k}(x,y)\xi^k(x,y)\right)X^j\frac {\partial}{\partial y^i}, 
\end{equation}
where we denote $G^i_{j k}(x,y) := \frac{\partial G_j^i(x,y)}{\partial y^k}$.

\bigskip

\subsubsection{Curvatures of Finsler manifolds}
\

\smallskip

\noindent
The \emph{curvature tensor} field
\begin{math}
  R\!= \! R^i_{jk}dx^j \!\otimes \! dx^k \! \otimes \!
  \frac{\partial}{\partial y^i}
\end{math}
has the expression
\begin{equation}
  \label{eq:R_coeff}
  R^i_{jk} =  \frac{\partial G^i_j(x,y)}{\partial x^k} 
  - \frac{\partial G^i_k(x,y)}{\partial x^j} + 
  G_j^m(x,y)G^i_{k m}(x,y) - G_k^m(x,y)G^i_{j m}(x,y). 
\end{equation} 
The \emph{Riemannian curvature tensor} is $R_y:=R(\cdot, y)$, its components
can be obtained as $R^i_{j}=R^i_{jk}y^k$. The Ricci curvature $Ric(y)$ is
defined to be the trace of $R_y$, $Ric(y) := R_m^m (x, y)$.  For a tangent
plane $P=Span\halmaz{y,u}\subset T_xM$, the flag curvature is defined as
\begin{displaymath}
  \mathbf{K}(P,y)=\frac{g_y(R_y(u),u)}{g_y(y,y)g_y(u,u)-g_y(y,u)^2}.
\end{displaymath}
If a manifold has \emph{constant flag curvature}
$\mathbf{K}=\lambda\in{\mathbb R}$, then the Ricci curvature is constant in
the sense that $Ric(y)=(n-1)\lambda F^2$ and the local expression of the
coefficients of the curvature is
\begin{displaymath}
  R^i_{jk} = \lambda\big(\delta_k^ig_{jm}(x,y)y^m -
  \delta_j^ig_{km}(x,y)y^m\big),
\end{displaymath}
where $\delta^i_j$ is the Kronecker delta symbol.  Assume that the Finsler
manifold $(M,\mathcal F)$ is locally projectively flat. Then for every point
$x\in M$ there exists an \emph{adapted} local coordinate system, that is a
mapping $(x^1,\dots ,x^n)$ on a neighbourhood $U$ of $x$ into the Euclidean
space $\mathbb R^n$\!, such that the straight lines of $\mathbb R^n$
correspond to the geodesics of $(M, \F)$. Then the \emph{geodesic
  coefficients} are of the form
\begin{equation}
  \label{eq:proj_flat_G_i}
  G^i \!=\! \mathcal P y^i, \quad  
  G^i_k \!=\! \frac{\partial\mathcal P}{\partial y^k}y^i \!
  + \! \mathcal P\delta^i_k,\quad G^i_{kl} 
  \!=\! \frac{\partial^2\mathcal P}{\partial y^k\partial y^l}y^i 
  \!+\! \frac{\partial \mathcal P}{\partial y^k}\delta^i_l \!+\! 
  \frac{\partial \mathcal P}{\partial y^l}\delta^i_k
\end{equation}
where $\mathcal P(x,y)$ is a 1-homogeneous function in $y$, called the
\emph{projective factor} of $(M,\F)$. According to Lemma 8.2.1 in
\cite[p.155]{Chern_Shen_2006}, if $(M \! \subset\! \mathbb R^n, \F)$ is a
projectively flat manifold, then its projective factor can be computed
using the formula
\begin{equation}
  \label{eq:P}
  \mathcal P(x,y) = \frac{1}{2\F}\frac{\partial\F}{\partial x^i}y^i.
\end{equation}

\bigskip

\subsubsection{Projectively flat Randers manifolds with constant curvature}
\label{sec:shen_proj_flat_randers} \

\smallskip

\noindent
Projectively flat Randers manifolds with constant flag curvature were
classified by Z. Shen in \cite{Shen_2003}. He proved that any projectively
flat Randers manifold $(M, \F)$ with non-zero constant flag curvature has
negative curvature. These metrics can be normalized by a constant factor so
that the curvature is $\lambda=-1/4$. In this case $(M, \F)$ is isometric to
the Finsler manifold defined by the Finsler function
\begin{equation}
  \label{eq:proj_F} 
  \F_a(x,y) = \frac{\sqrt{|y|^2 - \left(|x|^2|y|^2 - 
        \langle x,y\rangle^2\right)}}{1 - |x|^2}   + \epsilon
  \left(\frac{\langle x,y\rangle}  {1 - |x|^2} 
    + \frac{\langle a,y\rangle}{1 + \langle a,x\rangle}\right) 
\end{equation}
on the unit ball $\mathbb D^n \subset \mathbb R^n$, where $a\in \mathbb R^n$
is any constant vector with $|a|<1$ and $\epsilon=\pm 1$ (\cite[Theorem
1.1]{Shen_2003}). We note that the restriction of any orthogonal
transformation $\phi\in \mathbb O(n, \R^n)$ on $\mathbb D^n$ does not change
the Finsler function \eqref{eq:proj_F}, therefore one can assume that
$a\in \mathbb R^n$ has the form $a=(a_1, 0, \dots, 0)$. We can consider
$(\D^n, \F_a)$ as the \emph{standard model} of projectively flat Randers
manifolds with non-zero constant flag curvature.

We remark that the computation of the coefficients of the associated
connection is relatively easy: according to Lemma 8.2.1 of
\cite[p.155]{Chern_Shen_2006}, the projective factor $\P(x,y)$ can be computed
by the formula \eqref{eq:P} which gives in the case \eqref{eq:proj_F}
\begin{equation}
  \label{eq:proj_P}
  \P(x,y) =  \frac12 \left(\!\!\frac{\epsilon\sqrt{|y|^2 - \left(|x|^2|y|^2 
          - \langle x,y\rangle^2\right)} + \langle x,y\rangle}{1 - |x|^2} 
    - \frac{\langle a,y\rangle}{1 + \langle a,x\rangle}\right).
\end{equation}
The geodesic coefficients and the connection coefficients can be computed from
\eqref{eq:proj_P} by using \eqref{eq:proj_flat_G_i}.

\bigskip

\subsection{Holonomy group and its tangent Lie algebras}
\

\smallskip

\noindent
The \textit{holonomy group} $\Hol_x(M)$ of the Finsler manifold $(M, \F)$ at a
point $x\in M$ is the subgroup of the diffeomorphism group
$\mathsf{Diff}^{\infty} (T_x M )$ generated by parallel translations along
piece-wise differentiable closed curves initiated and ended at the point
$x \in M$. Since the canonical parallel translation is homogeneous and
preserves the Finsler function, the holonomy group can be considered on the
\emph{indicatrix}
\begin{math}
  \mathcal{I}_x:= \halmazvonal{ y \in T_x M}{\mathcal{F} (y) = 1}.
\end{math}
That way $\Hol_x(M)$ is a subgroup of the diffeomorphism group
$\mathsf{Diff} ^{\infty } (\I_x)$.  The topological closure
$\overline{\Hol_x\!(M)}$ of the holonomy group in the Fréchet topology of
$\mathsf{Diff}^{\infty } (\I_x)$ is called the \textit{closed holonomy group}.

\subsubsection{The curvature algebra}
\

\smallskip

\noindent
A vector field $\xi\in \X{TM}$ is called a \emph{curvature vector field} of
$(M, \mathcal{F} )$ if it is in the image of the curvature tensor, that is
$\xi = R(X,Y)$ for some $X,Y \in \X{M}$.  The \textit{curvature algebra}
$\mathfrak{R} (M)$ is the Lie algebra generated by curvature vector
fields. Its restriction
\begin{displaymath}
  \mathfrak{R}_x (M):= \halmazvonal{\xi |_{\I _x} }{ \xi \in
    \mathfrak{R} (M)}
\end{displaymath}
is the curvature algebra at the point $x\in M$.  $\mathfrak{R}_x (M)$ is a Lie
subalgebra of $\mathfrak{X} ^{\infty } (\I_x)$.

\subsubsection{The infinitesimal holonomy algebra}
\label{sec:infin-holon-algebra}
\

\smallskip

\noindent
The \textit{infinitesimal holonomy algebra}, denoted by $\mathfrak{hol}^* (M)$
is the smallest Lie algebra generated by curvature vector fields and by the
horizontal Berwald covariant derivation. More precisely,
$\mathfrak{hol}^* (M)$ is the smallest Lie algebra of vector fields on $TM$
satisfying the following two conditions:
\begin{packed_enumerate}
\item for any curvature vector field $\xi $ we have
  $\xi \in \mathfrak{hol}^*(M)$,
\item if $\xi \in \mathfrak{hol}^*(M)$ and $X \in \mathfrak{X} ^{\infty } (M)$
  then $\nabla _X \xi \in \mathfrak{hol} ^* (M)$.
\end{packed_enumerate}
By considering the  restriction of $\mathfrak{hol} ^* (M)$ on the indicatrix
$\I_x$ we obtain the infinitesimal holonomy algebra at the point $x \in M$: 
\begin{displaymath}
  \ihm x := \halmazvonal{ \xi |_{\I_x }}{ \xi \in \mathfrak{hol} ^* (M)}.
\end{displaymath}

We remark that if the manifold is two-dimensional, then the curvature algebra
$\mathfrak{R}_x(M)$ is at most one-dimensional, but the infinitesimal holonomy
algebra $\ihm x$ can be higher, even infinite dimensional: $\ihm x$ is
generated by the restriction of the curvature vector field and its covariant
derivatives.  Using the notation
$\xi_0=R(\partial_{x_1}, \partial_{x_2})\big|_{\I_{x}}$ and
\begin{math}
  \nabla_{i_1,\dots, i_k}\xi_0:=(\nabla_{\partial_{x_{i_1}}} \dots
  \nabla_{\partial_{x_{i_k}}}\xi)\big|_{\I_{x}}
\end{math}
we get
  \begin{equation}
    \label{eq:ih_x_0}
    \ihm{x}=
    \Bigl\langle
    \xi_0, \nabla_1\xi_0, \nabla_2\xi_0 , \nabla_{11}\xi_0,   \dots
    \Bigr\rangle_{\mathcal{L}ie}
\end{equation}
\textbf{Property} (Theorem 4, \cite{muzs_nagy_2011}): At any point $x \in M$
the infinitesimal holonomy algebra $\ihm{x}$ is tangent to the holonomy group
$\Hol_x (M)$.

\bigskip\bigskip \bigskip

\section{Holonomy of projectively flat Randers two-manifolds of non-zero
  constant curvature}

Our aim is to describe the holonomy structure of projectively flat
non-Riemannian Randers two-manifolds with non-zero constant flag curvature. As
a first step, we investigate the holonomy of the standard model described in
Subsection \ref{sec:shen_proj_flat_randers}.

\bigskip

Let $(\D^2, \F_a)$ be the Finsler two-manifold where $\mathbb D^2$ is the unit
ball in $\mathbb R^2$ and $\F_a$ is the Finsler function given by
\eqref{eq:proj_F} where $a=(a_1, 0)\in \mathbb R^2$ is a nonzero constant
vector with $|a_1|<1$. We have the following
\begin{proposition}
  \label{prop:D_2}
  The holonomy group of $(\D^2, \F_a)$ is maximal and $\overline{\Hol_x\!(M)}$
  is diffeomorphic to $\difS$.
\end{proposition}
\textit{Proof.} We consider the case when $\epsilon=+1$ in the expression
\eqref{eq:proj_F} of $\F_a$. The computation when $\epsilon=-1$ is analogous.
The projective factor $\P$ and the geodesic coefficients $G^i_j$ can be easily
computed by formula \eqref{eq:proj_P} and \eqref{eq:proj_flat_G_i}.  The
expression of the curvature vector field
\begin{math}
  \xi=R \left(\partial_{x_1}, \partial_{x_2}
  \right)
\end{math}
at the point $0\in \R^2$ is
\begin{equation}
  \label{eq:R_example}
  \xi=R
  \left(
    \frac{\partial}{\partial x_1},
    \frac{\partial}{\partial x_2}
  \right)=
  \frac{1}{4} \frac {y_{{2}} \left( a_1y_{{1}}  +\norm{y} \right) }
  {\norm{y}}  \frac{\partial}{\partial y_1}
  -\frac{1}{4}  \left(y_{1}+y_1a_1^2 +2\, a_1\norm{y}  \right)
  \frac{\partial}{\partial y_2}.
\end{equation}
Since the Minkowski norm at $0\in \D^2$ is
\begin{math}
  \mathcal{F}_a(0,y)=\norm{y}+ \langle a,y \rangle,
\end{math}
the indicatrix $\I_0\subset T_0M$ at $0$ is defined by the equation
$\sqrt{y_1^2+y_2^2} +a_1y_1=1$.  Using polar coordinates $(r, t)$ on
$T_0\R^2$, the equation of the indicatrix $\I_0$ is $r(1 + a_1 \cos t)=1$. A
parametrization of $\I_0$ is given by
\begin{equation}
  \label{eq:diff_s_ind}
  \phi(t)=  \bigl(  (y_1(t), y_2(t)  \bigr) =
  \left(
    \frac{\cos t}{1+ a_1 \cos t},  \frac{\sin t}{1 + a_1 \cos t}
  \right),
\end{equation}
in terms of the parameter $t$. Using this parametrisation the coordinate
expression of the restriction $\xi_0:= \xi\big|_{\I_0}$ of the curvature
vector field \eqref{eq:R_example} on the indicatrix $\I_0$ is
\begin{equation}
  \label{eq:xi_0}
  \xi_0:=  \omega(t) \frac{d}{dt} 
\end{equation}
where
\begin{equation}
  \label{eq:omega}
  \omega(t) :=-\tfrac{1}{4} (1+a_1\cos t)^2.
\end{equation}
Let us introduce the notation
\begin{equation}
  \label{eq:lemma_ih_0}
  \Sigma_n:= \mathrm{Span}_\R\halmazvonal{\xi^{l,m}_0
  }{l+ m \leq n},
\end{equation}
where
\begin{equation}
  \xi^{l,m}_0 :=(\sin^l  t \cos^m t) \cdot \xi_0 \in \X{\I_0}
\end{equation}
are vector fields on the indicatrix $\I_0$ defined as functional multiples of
\eqref{eq:xi_0}.  We have the following


\begin{lemma}
  \label{lemma:shen_renders}
  For any $n \in \N$ we have $\Sigma_n \subset \ih 0$.
\end{lemma}


\begin{remark}
  \label{remark:pyt}
  Using the the Pythagorean trigonometric identity $\sin^2 t + \cos^2 t = 1$,
  every elements of $\Sigma_{n}$ can be expressed as a linear combination of
  the elements $\xi_0^{0,m}\! = \! \cos^m t \xi_0$ and
  $\xi_0^{1,m-1}\!= \!\sin t \cos^{m-1} t \, \xi_0$, $0 \leq m \leq n$, that
  is
  \begin{equation}
    \label{eq:ind_pyt}
    \Sigma_n= \mathrm{Span}_\R \halmazvonal{\xi_0^{0,m},   \
      \xi_0^{1,m-1}}{ 0 \leq m \leq n\,} .
  \end{equation}
\end{remark}

\smallskip

\begin{proof}[Proof of the Lemma.]
  Taking into account Remark \ref{remark:pyt} we prove the lemma by
  mathematical induction by showing that the generating elements
  \eqref{eq:ind_pyt} of $\Sigma_n$ are elements of $\ih 0$.

  $\bullet$ \emph{First step.} From the definition of the infinitesimal
  holonomy algebra (see subsection \ref{sec:infin-holon-algebra}) we know that
  $\xi^{0,0}_0\!=\!\xi_0$ given by \eqref{eq:xi_0} is an element of $\ih 0$.
  Moreover, as \eqref{eq:ih_x_0} shows, the restriction of successive
  covariant derivatives of the curvature vector field \eqref{eq:R_example} on
  $\I_0$ are also elements of $\ih 0$. They can be expressed in terms of
  multiples of \eqref{eq:xi_0}. Computing the first covariant derivatives we
  find that
  \begin{subequations}
    \label{eq:cov_der_1_R} 
    \begin{alignat}{1}
      \label{eq:cov_der_R_1}
    (\nabla _1 \xi)\big|_{\I_0} &= -\tfrac{3}{2} \left( a_1 - \cos t \right) \xi_0,
    \\
    \label{eq:cov_der_R_2}
    (\nabla _2 \xi)\big|_{\I_0} &= \hphantom{-}\tfrac{3}{2} \sin t \, \xi_0,
  \end{alignat}
\end{subequations} 
Using a linear combination of \eqref{eq:cov_der_R_1} and
\eqref{eq:cov_der_R_2} we get that $\xi^{1,0}_0=\sin t \, \xi_0$ and
$\xi^{0,1}_0=\cos t \, \xi_0$ are element of $\ih 0$.  Therefore we have
\begin{equation}
  \label{eq:lemma_ih_0_1}
  \Sigma_1=\halmaz{  \xi_0, \ \sin  t \, \xi_0, \  \cos t  \, \xi_0 } \subset \ih{x_0},
\end{equation}
that is the statement of the Lemma is correct for $n=1$.

$\bullet$ \emph{Second step.}  By definition, the second covariant derivatives
of the curvature vector field are also elements of the infinitesimal holonomy
algebra. Computing them we can find that
\begin{subequations}
  \label{eq:cov_der_2_R} 
  \begin{alignat}{1}
    \label{eq:cov_der_R_11}
    (\nabla _1 \nabla_1 \xi) \big|_{\I_0}&= \hphantom{-}\tfrac{3}{4}\left(
      5a_1^2 - a_1\cos ^3 t -5a_1 \cos t +3 \cos ^2 t+1\right) \xi_0 ,
    \\
    \label{eq:cov_der_R_12}
    (\nabla_1 \nabla_2 \xi) \big|_{\I_0}&= -\tfrac{3}{4}\left( a_1\cos ^2 t
      -3a_1 -4 \cos t \right) \sin t \, \xi_0,
    \\
    \label{eq:cov_der_R_22}
    (\nabla_2 \nabla_2 \xi) \big|_{\I_0}&= \hphantom{-} \tfrac{3}{4}\left(
      a_1\cos ^3 t +5 -4 \cos ^2 t \right) \xi_0.
  \end{alignat}
\end{subequations} 
Using linear combinations of the elements \eqref{eq:lemma_ih_0_1} of
$\Sigma_1$ and \eqref{eq:cov_der_R_11}--\eqref{eq:cov_der_R_22} we get that
\begin{math}
  \label{eq:lemma_ih_0_2} 
  \halmaz{\cos^2  t \, \xi_0, \  \sin t \cos t  \, \xi_0} \subset \ih{x_0}.
\end{math}
Completing this set with the elements of \eqref{eq:lemma_ih_0_1} we get that
\begin{displaymath}
  \Sigma_2\subset \ih 0.
\end{displaymath}

$\bullet$ \emph{Third step.} Let us suppose that the statement of the lemma is
true for some $n \in \N$, that is $\Sigma_n \subset \ih 0$.  We will show that
$\Sigma_{n+1} \subset \ih 0$ too. According to the Remark, $\Sigma_{n+1}$ is
generated by the elements
\begin{equation}
  \label{eq:sigma_n+1}
  \halmazvonal{\xi_0^{0,m},   \
    \xi_0^{1,m-1}}{ 0 \leq m \leq n\,} \cup \halmaz{\xi^{0,n+1}_0, \ \xi^{1,n}_0}
\end{equation}
One can observe that the vector fields of the first set are elements of
$\Sigma_n$, and by the inductive hypotheses, they are elements of $\ih 0$.
Hence, to prove the lemma we have to show that $\xi^{0,n+1}_0$ and
$\xi^{1,n}_0$ are elements of $\ih 0$.  We have
\begin{alignat*}{1}
  \bigl[\xi_0, \xi^{0,n-1}_0 \bigr]
  & = \bigl[\xi_0 , \cos^{n-1} t \,\xi_0 \bigr]
  = \bigl( \mathcal{L}_{\xi_0}(\cos^{n-1} t) \bigr) \,\xi_0
  \\
  & = -(n-1)\omega(t)(\sin t \cos^{n-2} t) \,\xi_0
  \stackrel{\eqref{eq:omega}}{=}
  \\
  & = \tfrac{n-1}{4} (1+2a_1\cos t+a_1^2 \cos^2 t)(\sin t \cos^{n-2} t)
  \,\xi_0
  \\
  & = \tfrac{n-1}{4} \, \xi_0^{1,n-2}+\tfrac{a_1(n-1)}{2} \,
  \xi_0^{1,n-1}+\tfrac{a^2_1(n-1)}{4} \, \xi_0^{1,n}.
\end{alignat*}
By the inductive hypothesis, the Lie bracket on the left hand side and the
first two terms in the last line are elements of $\ih 0$. Consequently the
last one must be also an element of $\ih 0$. Moreover, the coefficients of
$\xi_0^{1,n}$ is a nonzero constant, therefore we get that
$\xi_0^{1,n} \in \ih 0$. 

Similarly, computing the Lie bracket of the elements $\xi_0$ and
$\xi_0^{1,n-2}$ of the Lie algebra $\ih 0$ we get
\begin{alignat*}{1}
  \bigl[\xi_0,& \xi^{1,n-2}_0 \bigr]
   = \bigl[\xi_0 , \sin t \cos^{n-2} t \,\xi_0 \bigr]
   = \mathcal{L}_{\xi_0}(\sin t \cos^{n-2} t) \,\xi_0
   \\
   & = \omega(t)(\cos t \cos^{n-2} t - (n-2)\sin^2t \cos^{n-3} t) \,\xi_0
   \stackrel{\eqref{eq:omega}}{=}
   \\
   & = \tfrac{n-3}{4}\xi^{0,n-3}_0 +\tfrac{a_1(n-2)}{2}\xi^{0,n-2}_0
   +\tfrac{a_1^2(n-2)-(n-1)}{4}\xi^{0,n-1}_0 -\tfrac{a_1(n-1)}{2}\xi^{0,n-1}_0
   -\tfrac{a_1^2(n-1)}{4}\xi^{0,n+1}_0
\end{alignat*}
From the inductive hypothesis we know that the Lie bracket on the left hand
side and the first four terms in the last line on the right hand side are
elements of $\ih 0$, therefore the last one must be also an element of
$\ih 0$.  Since the coefficient of $\xi_0^{0,n+1}$ is nonzero we get that
$\xi_0^{0,n+1} \in \ih 0$.  Consequently, the vector fields
\eqref{eq:sigma_n+1} generating $\Sigma_{n+1}$ are all elements of $\ih 0$ and
$\Sigma_{n+1}\subset \ih 0$.
\end{proof}


\begin{proof}[Proof of Proposition \ref{prop:D_2}.]
  From the multiple-angle formulas of the sine and cosine functions
  \begin{alignat*}{2}
    \sin nt &= \sum_{k=0}^{n} \tbinom{n}{k} \cos ^k t \sin^{n-k} t \sin
    \left(\tfrac{n-k}{2} \pi \right),
    \qquad
    \cos nt &= \sum_{k=0}^{n} \tbinom{n}{k} \cos^k t \sin^{n-k} t \cos
    \left(\tfrac{n-k}{2} \pi \right),
  \end{alignat*}
  we get that the vector fields
  $\sin nt \, \xi_0, \cos nt \, \xi_0 \in \X{\I_0}$ can be expressed as a
  linear combination of elements of $\Sigma_n$:
  \begin{subequations}
    \label{eq:sin_cos_nt}
    \begin{alignat}{1}
      \label{eq:eq:sin_nt}
      \sin nt\, \xi_0 &= \sum_{k=0}^{n} \tbinom{n}{k} \sin \left(\tfrac{n-k}{2}
        \pi \right) \xi_0^{k,n-k},
      \\
      \label{eq:eq:cos_nt}
      \cos nt \, \xi_0 &= \sum_{k=0}^{n} \tbinom{n}{k} \cos
      \left(\tfrac{n-k}{2} \pi \right)\, \xi_0^{k,n-k}.
    \end{alignat}
  \end{subequations}
  On the other side, Lemma \ref{lemma:shen_renders} shows that the vector
  fields of $\Sigma_n$ are elements of the holonomy algebra $\ih 0$.
  Therefore, the vector fields \eqref{eq:eq:sin_nt} and \eqref{eq:eq:cos_nt},
  $n \in \N$, are element of $\ih 0$ and
  \begin{equation}
    \mathrm{Span}\halmazvonal{ \sin nt\, \xi_0, \ \cos nt\, \xi_0}{
      n=0,1,\dots} \subset \ih 0.
  \end{equation}
  Moreover, any $2\pi$ periodic smooth function can be approximated uniformly
  by the arithmetical means of the partial sums of its Fourier series
  \cite[Theorem 2.12]{Katznelson_1976}. In particular the functions
  $(\sin nt)/\omega(t)$ and $(\cos nt)/\omega(t)$ can be approximated
  uniformly by their Fourier sums. Hence we get
  \begin{equation}
    \label{eq:fourier_alg_I}
    \big\{
    \tfrac{d}{dt},\, \cos nt\tfrac{d}{dt},\, \sin nt\tfrac{d}{dt}
    \big\}_{n\in \mathbb N} \subset  \overline{
      \left\{
        \mathrm{Span}    \big\{\xi_0,\, \cos nt\, \xi_0,\, 
        \sin nt\, \xi_0\big\}_{n\in
          \mathbb N}
      \right\}} \subset  \ih 0.
  \end{equation}
  The Lie algebra generated by the left hand side of \eqref{eq:fourier_alg_I}
  is diffeomorphic to the \emph{Fourier algebra} $\mathsf{F}(\mathbb S^1)$ on
  $\mathbb S^1$. Therefore from \eqref{eq:fourier_alg_I} we get that the
  infinitesimal holonomy algebra contains a Lie algebra diffeomorphic to the
  \emph{Fourier algebra} $\mathsf{F}(\mathbb S^1)$ on $\mathbb S^1$. Hence
  from Proposition 5.1 of \cite{MuzsNagy_max_2015} we get that the holonomy
  group $\mathsf{Hol}_0(\D^2)$ is maximal and 
  \begin{math}
    \overline{\mathsf{Hol}_0(\D^2)}
  \end{math}
  is diffeomorphic to $\mathsf{Diff}_+^{\infty}(\mathbb S^1)$.
\end{proof}

\bigskip

Using Z.~Shen's classification theorem of Randers manifolds we can get the
following


\begin{theorem}
  \label{thm:holon-proj-flat}
  The holonomy group of a simply connected non-Riemannian projectively flat
  Finsler two-manifold of constant non-zero flag curvature is maximal and
  $\overline{\Hol (M)}$ is diffeomorphic to the orientation preserving
  diffeomorphism group of $\mathbb S^1$, that is 
  \begin{displaymath}
    \overline{\Hol (M)}  \cong   \difS.
  \end{displaymath}
\end{theorem}

\begin{proof}
  Let $(M, \F)$ be a simply connected non-Riemannian projectively flat Finsler
  two-manifold of constant non-zero flag curvature and $x_0\in M$.  Since
  rescaling the metric by a constant factor does not change the connection and
  the parallel translation, it does not change the holonomy group
  either. Hence we can suppose that the metric is normalized so that the
  curvature is $\lambda=-\tfrac{1}{4}$.  Using Shen's results, $\F$ can be
  locally expressed in the form $\F_a$ given in \eqref{eq:proj_F} where
  $a=(a_1, 0)\in \mathbb R^2$ is a nonzero constant vector with $|a_1|<1$.
  From Proposition \ref{prop:D_2} we get, that the closed holonomy group of
  $(\D^2, \F_a)$ is maximal and diffeomorphic to
  $\mathsf{Diff}_+^{\infty}(\mathbb S^1)$, therefore the same is true for the
  closed holonomy group $\overline{\Hol_x\! (M)}$ of $(M, \F)$ at $x_0\in M$.
\end{proof}

\medskip

We can obtain the following classification:

\begin{corollary}
  \label{cor:holonomy}
  The closure of the holonomy group $\Hol (M)$ of a simply connected, locally
  projectively flat Randers two-manifold of constant flag curvature $\lambda$
  is
  \begin{packed_enumerate}
  \item the trivial group $\{id\}$, when $\lambda=0$;
  \item the rotation group $SO(2)$, when $\lambda\neq 0$ and the metric is
    Riemannian;
  \item the orientation preserving diffeomorphism group of the circle
    $\mathsf{Diff}_+^{\infty}(\mathbb S^1)$, when $\lambda\neq 0$ and the
    metric is non-Riemannian.
  \end{packed_enumerate}
\end{corollary}

\begin{proof}
  The holonomy structure of projectively flat Finsler manifolds was
  investigated in \cite{MuzsNagy_2015}: It has been proved that the holonomy
  group of projectively flat Finsler manifold is \emph{a)} finite dimensional
  if $\lambda=0$ or the metric is Riemannian, and \emph{b)} infinite
  dimensional if $\lambda\neq 0$ and the metric is non-Riemannian.  It is
  clear that the holonomy structures listed in \emph{(1)} and \emph{(2)}
  correspond to the (already well known) finite dimensional holonomy cases.
  Moreover, when $\lambda\neq 0$ and the metric is non-Riemannian we get
  \emph{(3)} from Theorem \ref{thm:holon-proj-flat}.
\end{proof}

\bigskip \bigskip\bigskip \bigskip

\end{document}